\documentclass[14pt,oneside,reqno]{amsart}
\usepackage{amsmath}
\usepackage{amssymb}
\usepackage{amscd}
\usepackage{epsfig}
\usepackage{amsxtra}
\usepackage{pifont}
\usepackage{mathabx}
\usepackage{MnSymbol}
\usepackage{accents}
\usepackage{scalerel,stackengine}
\stackMath
\newcommand\reallywidehat[1]{%
\savestack{\tmpbox}{\stretchto{%
  \scaleto{%
    \scalerel*[\widthof{\ensuremath{#1}}]{\kern-.6pt\bigwedge\kern-.6pt}%
    {\rule[-\textheight/2]{1ex}{\textheight}}
  }{\textheight}%
}{0.5ex}}%
\stackon[1pt]{#1}{\tmpbox}%
}

\newcommand\reallywidecheck[1]{%
\savestack{\tmpbox}{\stretchto{%
  \scaleto{%
    \scalerel*[\widthof{\ensuremath{#1}}]{\kern-.6pt\bigwedge\kern-.6pt}%
    {\rule[-\textheight/2]{1ex}{\textheight}}
  }{\textheight}%
}{0.5ex}}%
\stackon[1pt]{#1}{\scalebox{-1}{\tmpbox}}%
}
\newcommand{\exend}{\hfill$\Diamond$}

\numberwithin{equation}{section}

\newcommand{\RR}{{\mathbb R}}

\newcommand{\ZZ}{{\mathbb Z}}
\newcommand{\CC}{{\mathbb C}}
\newcommand{\BB}{{\mathbb B}}

\newcommand{\XX}{\mathbb X}
\newcommand{\AAA}{\mathbb A}

\newcommand{\cM}{{\mathcal M}}

\newcommand{\cA}{{\mathcal A}}
\newcommand{\cB}{{\mathcal B}}
\newcommand{\cC}{{\mathcal C}}

\newcommand{\dd}{\mbox{\rm d}}
\newcommand{\sgn}{\mbox{\rm sgn}}
\newcommand{\mm}{\mathsf{m}}
\newcommand{\eps}{\varepsilon}

\newcommand{\Cc}{C_{\mathsf{c}}}

 \newtheorem{theorem}{Theorem}[section]
 \newtheorem{lemma}[theorem]{Lemma}
 \newtheorem{proposition}[theorem]{Proposition}
 \newtheorem{corollary}[theorem]{Corollary}

 \newtheorem{definition}[theorem]{Definition}
 \newtheorem{example}[theorem]{Example}
  \newtheorem{remark}[theorem]{Remark}

\begin{document}
\title[Continuous diffraction]{Continuous diffraction spectrum and the uniform vanishing of Fourier--Bohr coefficients}

\author{Nicolae Strungaru}
\address{Department of Mathematics and Statistics, MacEwan University \\
10700 -- 104 Avenue, Edmonton, AB, T5J 4S2, Canada\\
and \\
Institute of Mathematics ``Simon Stoilow''\\
Bucharest, Romania}
\email{strungarun@macewan.ca}
\urladdr{https://sites.google.com/macewan.ca/nicolae-strungaru/home}

\begin{abstract}
In this paper we review the connection among continuity of the diffraction spectrum, the (uniform) vanishing of the Fourier--Bohr coefficients and the so called consistent phase frequency.
\end{abstract}

\maketitle

\section{Introduction}

A long standing question in the area of Aperiodic Order is to characterize all measures with pure point or continuous diffraction spectrum, respectively.

Inspired by the discovery of quasicrystals in the 1980's \cite{She}, the question of which measures have pure point spectrum is now well understood.
Over the years, various characterisations have been provided. Under various assumptions, pure point diffraction has been characterized in terms of almost periodicity of the
so called autocorrelation measure in \cite{ARMA,BM,MoSt}, in terms of pure point dynamical spectrum \cite{LMS,BL,Gou} and in terms of compactness of the hull in various topologies \cite{Gou2,MS}.
More recently pure point diffraction has been shown to be equivalent to mean almost periodicity of the underlying structure \cite{LSS}. This result generalizes partial work in this direction by \cite{BM,Gou2,Que}.

The so called Consistent Phase Property(CPP) plays an important role in the study of diffraction, especially in the pure point case, but not only. The consistent phase property simply asks if the intensity of the Bragg peak at position $\chi \in \widehat{G}$ is the same as the absolute value square of the Fourier--Bohr coefficient at $\chi$ (see below for details). More precisely, given a translation bounded measure $\mu$, with autocorrelation $\gamma$ along a van Hove sequence,
we say that the (CPP) holds at $\chi$ if the Fourier--Bohr coefficient $a_\chi^\cA(\mu)$ exists at $\chi$ and satisfies
\begin{equation*}
\reallywidehat{\gamma}(\{\chi\}) \, =\,  \left| a_\chi^\cA(\mu)\right|^2  \qquad \,.\tag{CPP}
\end{equation*}
For $G=\RR^d$, it was shown by Hof \cite{HOF} that the uniform existence of the Fourier--Bohr coefficients implies the (CPP), and later the result was shown in general. While the (CPP) may fail for individual measures, it holds almost surely for ergodic dynamical systems \cite{DL}, and it does hold for all elements in the dynamical system if the system is uniquely ergodic and the corresponding eigenfunction can be chosen to be continuous \cite{DL}.

The (CPP) also plays an important role for measures with pure point diffraction, as in this case it is the identifying feature of Besicovitch almost periodicity \cite{LSS}.

\smallskip

On another hand, systems with continuous diffraction spectrum are less understood. Even the concept of continuous diffraction spectrum does not seem to be precisely defined. While continuous diffraction spectrum usually means that the diffraction measure $\widehat{\gamma}$ is a continuous measure (or equivalently that there are no Bragg peaks), when studying Delone sets, or more generally positive measures of positive lower density, the diffraction always shows a trivial Bragg peak in the origin. In this situation, some people understand the diffraction to be continuous if the trivial Bragg peak is the only Bragg peak in the spectrum. This is the case for example in the Pinwheel tiling \cite{MPS}, as well as for various random point processes \cite[Example~ 11.6, Example~11.7,Theorem~11.4,Theorem~11.6]{TAO} (compare \cite{BBG,BBM}). To differentiate between the two notions, we will refer to the diffraction without the Bragg peak in the origin as the Bartlett diffraction spectrum (see Definition~\ref{def-Bart} and Remark~\ref{rem-Bart} below). In particular, we will say that the \emph{diffraction spectrum is continuous} if $\widehat{\gamma}$ is continuous, and that the \emph{Bartlett diffraction spectrum is continuous} if $\widehat{\gamma}$ contains at most one Bragg peak in the origin.

It follows immediately from \cite{ARMA,MoSt} that the diffraction spectrum is continuous if and only if the autocorrelation is a so-called null-weakly almost periodic measure (see \cite{ARMA,MoSt} for definition and properties), and that
the Bartlett diffraction spectrum is continuous if and only if there exists a constant $C$ (the intensity of the Bragg peak in the origin) such that $\gamma-C \theta_G$ is a null-weakly almost periodic measure.

Similarly to the characterisation of pure point measures, one would like to characterize continuous (Bartlett) diffraction spectrum in terms of properties of the underlying structure, and not via properties of the autocorrelation measure. There are some results in this direction, usually connecting the continuous (Bartlett) diffraction spectrum with vanishing of the Fourier--Bohr coefficients, which are often hidden in the literature or are unstated immediate consequences of other results. It is our goal here to review  in a systematic way the connection between the continuity of the (Bartlett) diffraction spectrum and vanishing of the Fourier--Bohr coefficients.

\smallskip
It turns out that the (CPP), continuity of (Bartlett) diffraction spectrum and the vanishing of the (non-trivial) Fourier--Bohr coefficients are strongly connected. Indeed, given a measure $\mu$ whose autocorrelation $\gamma$ exists along a sequence $\cA$, consider the following three properties:

\begin{itemize}
  \item[(A)]  The (Bartlett) diffraction spectrum continuous .
  \item[(B)]  $a_\chi^\cA(\mu)$ exists and is zero for all $\chi \in \widehat{G}$ (respectively for all $0 \neq \chi \in \widehat{G}$ ).
  \item[(C)]  (CPP) holds for all $\chi \in \widehat{G}$ (respectively for all $0 \neq \chi \in \widehat{G}$ ).
\end{itemize}

Then, any two of these conditions imply the third.

In particular, assuming the (CPP), we get the equivalence between (A) and (B). While this looks like a nice result, in general, it does not seem easy to decide when the (CPP) holds for an arbitrary measure $\mu$, especially when the diffraction is not pure point.

There are only two general results we know on the (CPP). The first one, going back to Hof, says that the (CPP) always holds if the Fourier--Bohr coefficients exist uniformly \cite{HOF,HOF1}. This is in turn equivalent to the continuity of eigenfunctions \cite{LSS}, and often fails.  The second result, due to Lenz, says that if $\mm$ is an ergodic measure then the (CPP) holds for $\mm$-almost all $\omega \in \XX(\mu)$. But unfortunately this does not tell us if (CPP) holds for our original measure $\mu$, and there are known examples where this is not the case.

\smallskip

More recently, it was shown in \cite{LS} that if there is no Bragg peak at $\chi \in \widehat{G}$ then the Fourier--Bohr coefficient exists and is zero. In particular, (A) always implies both (B) and (C). Another interesting consequence of this result is that (A) holding for \textbf{all} diffraction measures is actually equivalent (B) holding uniformly (see Proposition~\ref{Prop1} below).

The main goal of this paper is to expand on this result. On the way, we obtain the following Theorem, which provides  more conditions that are equivalent to uniform vanishing of the Fourier--Bohr coefficient at a single character $\chi$.

\smallskip

\textbf{Theorem}~\ref{T1} Let $\mu \in \cM^\infty(G)$ and let $\cA$ be a van Hove sequence. Let $\chi \in \widehat{G}$. Then, the following are equivalent:
\begin{itemize}
\item[(i)] The Fourier--Bohr coefficient $a_\chi^{\cA}(\mu)$ are uniformly vanishing along $\cA$.
\item[(ii)] For all van Hove sequence $\cB$, Fourier--Bohr coefficient $a_\chi^{\cB}(\mu)$ is uniformly vanishing.
\item[(iii)]  The Fourier--Bohr coefficient $a_\chi^{\cA}(\omega)$ is uniformly vanishing for all $\omega \in \XX(\mu)$.
\item[(iv)] For each van Hove sequence $\cB$ the following Fourier--Bohr coefficient exists and
\[
a_\chi^\cB(\mu)\,=\,0 \,.
\]
\item[(v)]If the Fourier--Bohr coefficient $a_\chi^\cB(\mu)$ exists along some van Hove sequence $\cB$ then
  \[
a_\chi^\cB(\mu)\,=\,0 \,.
\]
\item[(vi)] For every autocorrelation $\gamma^{}_{\mu}$ of $\mu$ with respect to some van Hove sequence $\cB$ we have
  \[
  \reallywidehat{\gamma^{}_{\mu}}(\{ \chi \})\,=\,0 \,.
  \]
\item[(vii)] For all $G$-invariant probability measures $\mm$ on $\XX(\mu)$ we have
\[
\reallywidehat{\gamma_{\mm}}(\{ \chi \})\,=\,0 \,.
\]
\item[(viii)] For all ergodic measures $\mm$ on $\XX(\mu)$ we have
\[
\reallywidehat{\gamma_{\mm}}(\{ \chi \})\,=\,0 \,. 
\]
\item[(ix)] For every $\omega \in \XX(\mu)$ and every autocorrelation $\gamma_{\omega}^{}$ of $\omega$ along some van Hove sequence $\cB$ we have
  \[
  \reallywidehat{\gamma^{}_{\omega}}(\{ \chi \})\,=\,0 \,.\tag*{$ \qed $}
  \]
\end{itemize}

\smallskip

As consequences, we obtain in Theorem~\ref{T11} and Theorem~\ref{T11} a characterisation for the continuity of all the (Bartlett) diffraction spectra of a measure $\mu$, which generalizes  \cite{LS3}.
A result in the spirit of Theorem~\ref{T11} and Theorem~\ref{T111}, but holding for one a fixed choice of the autocorrelation/van Hove sequence would be more interesting and helpful, but seems to be out of our reach for now.

\section{Preliminaries}

Throughout this paper $G$ is a second countable LCAG. As usual we denote by $\Cc(G)$ the space of compactly supported continuous functions on $G$ and by $\cM(G)$ the space of (Radon) measures on $G$.

For a function $f$ and some $t \in G$ we denote by $\tau^{}_tf$ the translation of $f$ by $t$, that is
\[
\tau^{}_t f(x) \,:=\, f(x-t) \,.
\]
We denote by $\tilde{\,}$ the involution
\[
\tilde{f}(t)\, := \,\overline{f(-t)} \,.
\]
Same way, for a measure $\mu \in \cM(G)$ and $t \in G$ we define
\begin{align*}
\tau^{}_t\mu(\varphi)\, &:=\, \mu(\tau_{-t}\varphi) \qquad \forall \varphi \in \Cc(G) \\
\tilde{\mu}(\varphi)\,&:=\, \overline{ \mu(\tilde{\varphi})} \qquad \forall \varphi \in \Cc(G)  \,.
\end{align*}

\subsection{Autocorrelation and  diffraction of a measure}

Recall here that a measure $\mu \in \cM(G)$ is called \emph{translation bounded} if for one (and hence all \cite{BM}) precompact Borel set $B$ with non-empty interior we have
\[
\| \mu \|^{}_{B} \, := \, \sup_{t \in G} \left| \mu \right|(t+B) <\infty \,.
\]
The space of translation bounded measures is denoted by $\cM^{\infty}_{}(G)$.

\smallskip
Next, recall \cite{Sch} that a sequence $\cA=(A^{}_n)^{}_n$ of pre-compact Borel subsets of $G$ is called a \emph{van Hove sequence} if for all compact sets $K \subseteq G$ we have
\[
\lim^{}_{n \to \infty} \frac{|\partial_{}^{K} A^{}_{n}|}{|A^{}_{n}|}  \, = \,  0
\]
where the \emph{$K$-boundary $\partial^K(A)$} of a set $A$ is defined as
\[
\partial^{K} A \, := \,  \bigl( \overline{A+K}\setminus A\bigr) \cup
\bigl((\left(G \backslash A\right) - K)\cap \overline{A}  \bigr) \,.
\]

Let us now recall the following definition.

\begin{definition} We say that the \emph{autocorrelation $\gamma^{}_\mu$} of $\mu \in \cM(G)$ exists along the van Hove sequence $\cA$ if in the vague topology we have
\[
\gamma^{}_\mu \,: =\,  \lim^{}_{n \to \infty} \frac{1}{|A_n|} \left( \mu|^{}_{A^{}_n} \right)* \widetilde{\left( \mu|^{}_{A^{}_n} \right)} \,.
\]
As the limit of positive definite measures, $\gamma^{}_\mu$ is positive definite, and hence it has a \emph{Fourier transform $\reallywidehat{ \gamma^{}_\mu}$} \cite{ARMA1,BF,MoSt} which is a positive measure.

The measure $\reallywidehat{ \gamma^{}_\mu}$ is called the \emph{diffraction spectrum of $\mu$ along $\cA$}.
\end{definition}

Let us also introduce the following definition, which was suggested to us by Michael Baake.

\begin{definition}\label{def-Bart}
The measure
\[
\Gamma^{}_{\mu} \, :=\,   \reallywidehat{ \gamma^{}_\mu} - \left(\reallywidehat{ \gamma^{}_\mu} (\{0\}) \delta_0 \right)
\]
is called the \emph{Bartlett diffraction spectrum of $\mu$}.
\end{definition}

\begin{remark}\label{rem-Bart} \cite[Remark~15]{BBM} In the case of random point processes, the measure $\sigma$ is exactly the so called Bartlett spectrum of the corresponding stationary random measure (compare \cite[Sect~5]{DV}, \cite{Gou}). \exend

\end{remark}

For translation-bounded measures, by eventually passing to a subsequence, we can always assume that the autocorrelation exists. Indeed, the following result can be found in many places, see for example \cite[Theorem 4.14(a)]{LSS3}.

\begin{lemma} Let $\mu \in \cM^\infty_{}(G)$ and let $\cA$ be a van Hove sequence. Then, the autocorrelation $\gamma_\mu$ exists along a subsequence of $\cA$. \qed
\end{lemma}

\subsection{Fourier--Bohr coefficients}

In this section we will briefly recall the Fourier--Bohr coefficients and their properties.

\begin{definition} Given a measure $\mu$, a van Hove sequence $\cA$ and $\chi \in \widehat{G}$, we say that the \emph{Fourier--Bohr coefficient $a_\chi^\cA(\mu)$} exists along $\cA$ if the following limit exists:
\[
a_\chi^\cA(\mu)\, :=\,\lim^{}_{n \to \infty} \frac{1}{|A_n|} \int^{}_{A_n} \overline{\chi(t)} \dd \mu(t) \,.
\]
We say that  the \emph{Fourier--Bohr coefficient $a_\chi^\cA(\mu)$ exists uniformly} along $\cA$ if the following limit exists uniformly in $s$:
\[
a_\chi^\cA(\mu)\, :=\, \lim^{}_{n \to \infty} \frac{1}{|A_n|} \int^{}_{s+A_n} \overline{\chi(t)} \dd \mu(t) \,.
\]
For a function $f \in L^1_{loc}(G)$ we say that the \emph{Fourier--Bohr coefficient $a_\chi^\cA(f)$ exists (uniformly)} along $\cA$ if the Fourier--Bohr coefficient $a_{\chi}^\cA(\mu)=: a_\chi^\cA(f)$ exists (uniformly), where $\mu= f \theta_G$ is the absolutely continuous measure with density function $f$.
\end{definition}

Given $\mu \in \cM^\infty(G)$ and a van Hove sequence $\cA$, then there exists a subsequence $\cB$ of $\cA$ along which all Fourier--Bohr coefficients exist \cite[Corollary~7.4]{LS}. Therefore, exactly as with the autocorrelation, the existence of the Fourier--Bohr coefficients can be assumed by changing the van Hove sequence.

Below we will need the following results of \cite{LSS}.

\begin{proposition}\cite[Corollary~1.11]{LSS}\label{Lem-FB} Let $\mu \in \cM^\infty_{}(G)$, $\cA$ be a van Hove sequence and $\chi \in \widehat{G}$. Then,
\begin{itemize}
  \item[(a)] If $a_{\chi}^\cA(\mu)$ exists then $a_{\chi}^\cA(\mu*\varphi)$ exists for all $\varphi \in \Cc(G)$ and
  \[
  a_{\chi}^\cA(\mu*\varphi)\, =\, \widehat{\varphi}(\chi) a_{\chi}^\cA(\mu) \,.
  \]
  \item[(b)] If $a_{\chi}^\cA(\mu)$ exists uniformly then $a_{\chi}^\cA(\mu*\varphi)$ exists uniformly for all $\varphi \in \Cc(G)$ and
  \[
  a_{\chi}^\cA(\mu*\varphi)\,=\,\widehat{\varphi}(\chi) a_{\chi}^\cA(\mu) \,.
  \]
  \item[(c)] If $a_{\chi}^\cA(\mu*\varphi)$ exists for some $\varphi \in \Cc(G)$ with $\widehat{\varphi}(\chi) \neq 0$ then $a_{\chi}^\cA(\mu)$ exists  and
  \[
  a_{\chi}^\cA(\mu*\varphi)\, =\, \widehat{\varphi}(\chi) a_{\chi}^\cA(\mu) \,.
  \]
   \item[(d)] If $a_{\chi}^\cA(\mu*\varphi)$ exists uniformly for some $\varphi \in \Cc(G)$ with $\widehat{\varphi}(\chi) \neq 0$ then $a_{\chi}^\cA(\mu)$ exists uniformly and
  \[
  a_{\chi}^\cA(\mu*\varphi)\, =\, \widehat{\varphi}(\chi) a_{\chi}^\cA(\mu) \,.\tag*{$ \qed $}
  \]
\end{itemize}
\end{proposition}

Next, let us recall that the uniform existence of Fourier--Bohr coefficients is equivalent to their existence along all van Hove sequences. Indeed, we have:

\begin{proposition}\cite[Corollary~1.12]{LSS}\label{cor2} Let $\mu \in \cM^\infty_{}(G)$, $\cA$ be a van Hove sequence and $\chi \in \widehat{G}$. Then, the following are equivalent:
\begin{itemize}
  \item[(i)] $a_{\chi}^\cA(\mu)$ exists uniformly.
  \item[(ii)] $a_{\chi}^\cB(\mu)$ exists for all van Hove sequences $\cB$.
  \item[(iii)] $a_{\chi}^\cB(\mu)$ exists uniformly for all van Hove sequences $\cB$.
\end{itemize}
Moreover, in this case, for all van Hove sequences $\cB$ we have
\[
a_{\chi}^\cB(\mu)\, = \, a_{\chi}^\cA(\mu) \,.
\]
\end{proposition}
\begin{proof}
The equivalence follows from \cite[Corollary~1.12]{LSS}. The last claim is immediate. Indeed, if $\cA=(A^{}_n)^{}_n$ and $\cB=(B^{}_n)^{}_n$, then so is $\cC=(C^{}_n)^{}_n$ where
\[
C_n\, := \,
\begin{cases}
A_m & \mbox{ if } n=2m  \\
B_m & \mbox{ if } n=2m+1  \,.
\end{cases}
\]
Since $\cA,\cB$ are subsequences of $\cC$ and $a_{\chi}^\cC(\mu)$ exists by (ii) then
\[
a_{\chi}^\cA(\mu)\, =\, a_{\chi}^\cC(\mu)\, =\, a_{\chi}^\cB(\mu) \,. \qedhere
\]
\end{proof}

Because of this result, whenever the Fourier--Bohr coefficient exists uniformly we write $a^{}_\chi(\mu)$ and not specify the van Hove sequence.

\medskip

We complete the section by recalling the following  results relating the Fourier--Bohr coefficients with the intensity of Bragg peaks. We start by recalling that the (CPP) holds at all points where the Fourier--Bohr coefficients exist uniformly. For $G=\RR_{}^d$ this result goes back to Hof \cite{HOF,HOF1}.

\begin{proposition} \cite[Theorem 4.1]{NS} Let $\mu$ be a measure and $\cA$ be a van Hove sequence along which the autocorrelation $\gamma^{}_\mu$ exists.

If $\chi \in \widehat{G}$ is so that the Fourier--Bohr coefficient $a^{}_\chi(\mu)$ exists uniformly, then
\[
\reallywidehat{\gamma^{}_\mu}(\{\chi\}) \, =\, \left| a^{}_\chi(\mu)\right|^2 \,. \tag*{$ \qed $}
\]
\end{proposition}

\smallskip Next, let us recall the so called the Bombieri--Taylor conjecture. It says that (CPP) automatically holds at all positions where there is no Bragg peak.

\begin{proposition}\cite[Proposition~7.1]{LS}\label{PropBT} Let $\mu \in \cM^\infty(G)$ and let $\cA$ be a van Hove sequence along which the autocorrelation $\gamma_{\mu}$ exists.

If $\chi \in \widehat{G}$ is so that $\reallywidehat{\gamma_\mu}(\{ \chi\})=0$. Then, the Fourier--Bohr coefficient $a_{\chi}^\cA(\mu)$ exists and is zero. In particular, the (CPP) holds at $\chi$.\qed
\end{proposition}

Let us note here the following consequence of the Bombieri--Taylor.

\begin{corollary}\label{cor3} Let $\mu \in \cM^\infty(G)$ and let $\cA$ be a van Hove sequence along which the autocorrelation $\gamma_{\mu}$ exists.
\begin{itemize}
  \item[(a)] If the diffraction measure $\reallywidehat{\gamma_{\mu}}$ is continuous, then the Fourier--Bohr coefficient $a_{\chi}^\cA(\mu)$ exists and is zero for all $\chi \in \widehat{G}$.
 \item[(b)] If the Bartlett diffraction measure $\Gamma_{\mu}$ is continuous, then the Fourier--Bohr coefficient $a_{\chi}^\cA(\mu)$ exists and is zero for all $0 \neq \chi \in \widehat{G}$. \qed
\end{itemize} 
\end{corollary}

\begin{remark} The converse implications do not hold in Corollary~\ref{cor3}. Indeed, consider the following measures on $\RR$:
\begin{align*}
  \mu \, &:= \, \sgn(x) \lambda \\
  \nu \, &:= \, \sgn(x) \Delta_{\ZZ} \,,
\end{align*}
where
\[
\sgn(x) \,:=\, \begin{cases}
            1, & \mbox{if } x>0 \\
            0, & \mbox{if } x=0 \\
            -1, & \mbox{otherwise} ,
          \end{cases}
\]
is the  sign function and $\lambda$ denotes the Lebesgue measure.

Let $A_n=[-n,n]$. Then,
\begin{itemize}
  \item[(a)] The autocorrelations $\gamma_{\mu}, \gamma_{\nu}$ exists with respect to $\cA=(A_n)_n$ and
  \begin{align*}
  \gamma_{\mu} \, &= \, \lambda\\
   \gamma_{\nu} \, &= \, \delta_{\ZZ}\,.
  \end{align*}
  In particular,
    \begin{align*}
 \reallywidehat{ \gamma_{\mu}} \, &=\,  \delta_0\\
 \reallywidehat{  \gamma_{\nu}} \, &=\, \delta_{\ZZ}\,.
  \end{align*}
  \item[(b)] The Bartlett diffraction spectrum of $\nu$ is
  \[
  \Gamma_{\nu} \, = \, \delta^{}_{\ZZ \backslash \{ 0\}} \,.
  \]
  \item[(c)] For all $y \in \RR=\widehat{\RR}$ the Fourier--Bohr coefficients $a_y^\cA(\mu), a_y^\cA(\nu)$ exist and
  \[
  a_y^\cA(\mu) \, = \, a_y^\cA(\nu) \,=\, 0 \,.\tag*{\hfill$\Diamond$}
  \]
\end{itemize}
\end{remark}

\smallskip

Finally, let us recall the following consequence of Bombieri--Taylor:

\begin{proposition}\cite[Corollary~7.8]{LS}\label{Prop1} Let $\mu \in \cM^{\infty}_{}$ and $\chi \in \widehat{G}$. Then, the following are equivalent:
\begin{itemize}
  \item[(i)] For every autocorrelation $\gamma^{}_{\mu}$ of $\mu$ along some van Hove sequence we have
  \[
  \reallywidehat{\gamma^{}_{\mu}}(\{ \chi \}) \, =\, 0 \,.
  \]
  \item[(ii)] If the Fourier--Bohr coefficient $a_{\chi}^\cB(\mu)$ of $\mu$ exists for some $\chi$ along some van Hove sequence $\cB$ then $a_{\chi}^\cB(\mu)=0$.
  \item[(iii)] The Fourier--Bohr coefficient $a_\chi(\mu)$ exists uniformly  for all $\chi \in \widehat{G}$ and $a_{\chi}(\mu)=0$. \qed
\end{itemize}
\end{proposition}

Since we will often run into this situation, let us introduce the following definition.

\begin{definition}Let $\mu \in \cM^{\infty}_{}$ and $\chi \in \widehat{G}$. We say that the Fourier--Bohr coefficient $a_\chi(\mu)$ is \emph{uniformly vanishing} if $a_\chi(\mu)$ exists uniformly and is zero.
\end{definition}

\subsection{Diffraction of dynamical systems of translation bounded measures}

For a fixed $C>0$ and precompact open set $V$, denote by
\[
\cM^{}_{C,V}(G) \,:=\,  \big\{ \mu \in \cM^{\infty}_{}(G) \, : \, \| \mu \|^{}_{V} \leq C \big\} \,.
\]
Then, $\cM^{}_{C,V}(G)$ is compact in the vague topology for measures \cite[Theorem~2]{BL}, and the vague topology is metrisable on this space \cite[Theorem~2]{BL}.

\smallskip
Now, given a measure $\mu \in \cM^\infty(G)$, we can find some $C,V$ so that
\[
O^{}_{\mu} \, := \, \big\{ \tau^{}_t \mu \, : \, t \in G \big \} \subseteq \cM^{}_{C,V}(G) \,.
\]
It follows that the vague completion
\[
\XX(\mu) \, :=\,  \overline{ \big\{ \tau^{}_t \mu : t \in G \big\}  }
\]
becomes a metrisable compact space. Since $G$ acts on $\XX(\mu)$ via translates, the pair $(\XX(\mu), G)$ becomes a topological dynamical system.

Next, each $\varphi \in \Cc(G)$ induces a continuous function $F^{}_\varphi \in \Cc(\XX(\mu))$ via
\[
F^{}_\varphi(\omega) \,:=\, \varphi*\omega(0) \, = \, \int^{}_{G} \varphi(-t) \, \dd \omega(t) \,.
\]
$F^{}_\varphi$ commutes with the translation, meaning \cite{BL}
\[
\tau^{}_t F^{}_\varphi\,=\, F^{}_{\tau^{}_t\varphi} \qquad \forall \varphi \in \Cc(G), t \in G \,.
\]

\smallskip
Let us recall the following result from \cite{BL}, see \cite{LSS} for the second claim.

\begin{theorem}\cite[Proposition~6, Lemma~7]{BL}\label{thm:BL} Let $\mm$ be any $G$-invariant probability measure on $\XX(\mu)$. Then, there exists a positive definite measure $\gamma^{}_\mm$ on $G$ such that, for all $\varphi \in \Cc(G)$ and $t \in G$ we have
\[
\gamma^{}_\mm*\varphi*\widetilde{\varphi}(t) \,=\, \langle F^{}_{\varphi}, \tau^{}_t F^{}_{\varphi} \rangle \,.
\]
Moreover, if $\omega \in \XX(\mu)$ is generic for $\mm$ along some van Hove sequence $\cA$, then the autocorrelation $\gamma^{}_{\omega}$ of $\omega$ exists along $\cA$ and
\[
\gamma^{}_\omega \,=\, \gamma^{}_\mm \,.\tag*{$ \qed $}
\]
\end{theorem}

This leads to the following definition:

\begin{definition} Let $\mu \in \cM^\infty_{}(G)$ and let $\mm$ be a $G$-invariant probability measure on $\XX(\mu)$. The measure $\gamma^{}_{\mm}$ is called the
 \emph{autocorrelation measure of the dynamical system $(\XX(\mu), \mm)$}.

Its positive Fourier transform $\reallywidehat{ \gamma^{}_\mm}$ is called the \emph{diffraction spectrum of the dynamical system $(\XX(\mu), \mm)$}.

Similarly with the diffraction measure, the measure
\[
\Gamma_{\mm}^{} \, := \,  \reallywidehat{ \gamma^{}_\mm} - \left(\reallywidehat{ \gamma^{}_\mm}(\{0\})\right) \delta_0
\]
is called the  \emph{Bartlett diffraction spectrum of the dynamical system $(\XX(\mu), \mm)$}
\end{definition}

\medskip
Let us now briefly recall the connection between the autocorrelation/diffraction of a measure and of the corresponding dynamical system. First recall the following result:

\begin{lemma} \cite[Proposition~4.5(b)]{LS3} \cite{LSS} Let $\mu \in \cM^\infty(G)$ and let $\cA$ be a van Hove sequence. Then, there exists some subsequence $\cB$ of $\cA$ and a $G$-invariant probability measure $\mm$ on $\XX(\mu)$
such that $\mu$ is generic for $\mm$ with respect to $\cB$.

In particular, the autocorrelation $\gamma_{\mu}$ of $\mu$ exists along $\cB$ and 
\[
\gamma^{}_{\mu} \, = \, \gamma^{}_{\mm} \,. \tag*{$ \qed $}
\] 
\end{lemma}

It follows that the set of autocorrelations $\gamma$ of $\mu$ along all van Hove sequences where they exists is the same as the set of all autocorrelations $\gamma^{}_{\mm}$ of $G$-invariant probability measures $\mm$ on $\XX(\mu)$ for which $\mu$ is generic.

\begin{corollary}[Autocorrelation of elements vs autocorrelation of dynamical systems] \label{Cor1} Let $\mu, \gamma \in \cM^\infty(G)$. Then $\gamma$ is the autocorrelation of $\mu$ (along some van Hove sequence $\cA$) if and only if there exists some $G$-invariant probability measure $\mm$ on $\XX(\mu)$ such that $\mu$ is generic for $\mm$ (along some van Hove sequence $\cB$) and
\[
\gamma\, = \, \gamma^{}_\mm \,.
\]
Moreover, in this case, $\cB$ can be chosen to be a subsequence of $\cA$. \qed
\end{corollary}

\begin{remark} Let us note here in passing that it is possible for two distinct $G$-invariant  measures $\mm \neq \mathsf{n}$ on $\XX(\mu)$ to satisfy
\[
\gamma^{}_{\mm} \, = \, \gamma^{}_{\mathsf{n}} \,.
\]
Indeed, The Rudin--Shapiro chain $\omega^{}_{\mathsf{RS}}$ (see \cite[Section~4.7.1]{TAO}) has autocorrelation $\gamma^{}_{\mathsf{RS}}$ and diffraction $\reallywidehat{\gamma^{}_{\mathsf{RS}}}$ given by \cite[Theorem~10.2]{TAO}
\begin{align*}
 \gamma^{}_{\mathsf{RS}}\,  &= \, \delta_0 \\
  \reallywidehat{\gamma^{}_{\mathsf{RS}}} \, & = \, \lambda \,.
\end{align*}
Moreover, as a substitution tiling, the hull $\XX(\omega^{}_{\mathsf{RS}})$ is uniquely ergodic \cite[Theorem~4.3]{TAO}. Let $\mm$ be the unique ergodic measure on this hull.

Next, consider the Bernoulli comb $W$ on the line from \cite[Sect.~11.2.1]{TAO} with probability $p=\frac{1}{2}$. Then, there exists an ergodic measure $\mathsf{n}$ on $\XX = \{ \pm 1 \}^{\ZZ}$ such that for $\mathsf{n}$ almost surely all $W \in \XX$ the diffraction of $W$ are given by \cite[Proposition~11.1.]{TAO}
\begin{align*}
 \gamma^{}_{W}\,  &= \, \delta_0 \\
  \reallywidehat{\gamma^{}_{W}} \, & = \, \lambda \,.
\end{align*}
Now, construct a translation bounded measure $\mu$ by gluing together alternately larger and larger patches around the origin of $\omega^{}_{\mathsf{RS}}\delta^{}_{\ZZ}$ and $W \delta^{}_{\ZZ}$, respectively. Then, both $\mm$ and $\mathsf{n}$ are $G$-invariant ergodic measures on $\XX(\mu)$ with the same autocorrelation measure. Moreover, there exists van Hove sequences $\cA$ and $\cB$ such that $\mu$ is generic for $\mm$ along $\cA$ and for $\mathsf{n}$ along $\cB$. \exend
\end{remark}

The following result is well known and follows immediately from Corollary~\ref{Cor1}.

\begin{corollary}\label{cor22} Let $\mu \in \cM^\infty_{}(G)$ be so that $\XX(\mu)$ is uniquely ergodic with uniquely ergodic measure $\mm$. Then, for all $\omega \in \XX(\mu)$ and all van Hove sequences $\cA$,  the autocorrelation $\gamma_\omega$ of $\mu$ exists along $\cA$ and
\[
\gamma_{\omega} \, = \, \gamma_{\mm} \,.\tag*{$ \qed $}
\]
\end{corollary}

\begin{remark} Let $\mu \in \cM^\infty(G)$. It is easy to see that the following are equivalent:
\begin{itemize}
  \item[(i)] The autocorrelation of $\mu$ exists along all van Hove sequences.
  \item[(ii)] The autocorrelation of $\mu$ exists and is the same along all van Hove sequences.
  \item[(iii)] If the autocorrelations $\gamma, \eta$ of $\mu$ exists along two van Hove sequences, then $\gamma=\eta$.
    \item[(iv)] For each $\omega \in \XX(\mu)$ and each van Hove sequence $\cA$ the autocorrelation $\gamma$ of $\omega$ exists along $\cA$.
\end{itemize}
In this situation we say that $\mu$ \emph{has unique autocorrelation}.

Corollary~\ref{cor22} says that unique ergodicity implies unique autocorrelation. The converse is not true. \exend
\end{remark}

\section{Continuous (Bartlett) diffraction spectra for ergodic dynamical systems of translation bounded measures}

In this section we review and expand on a result of Lenz which shows that the (CPP) holds almost surely for  ergodic dynamical systems of translation bounded measures. As an application, this provides the first general connection between continuous diffraction and vanishing of Fourier--Bohr coefficients.

\begin{theorem}\cite[Theorem~5(b)]{DL}\label{Thm-BTDS} Let $\mu \in \cM^\infty(G)$, let $\cA$ be a van Hove sequence along which the Birkhoff ergodic theorem holds, let $\chi \in \widehat{G}$ and let $\mm$ be an ergodic measure on $\XX(\mu)$. Then, for almost all $\omega \in \XX(\mu)$ the Fourier--Bohr coefficient $a_\chi^\cA(\omega)$ exists and satisfies
\[
\reallywidehat{\gamma^{}_{\mm}}(\{\chi \}) \, = \, \left| a_\chi^\cA(\omega) \right|^2 \,.\tag*{$ \qed $}
\]
\end{theorem}

As an immediate consequences, we get:

\begin{corollary}\label{cor4} Let $\mu \in \cM^\infty(G)$, let $\cA$ be a van Hove sequence along which the Birkhoff ergodic theorem holds and let $\mm$ be an ergodic measure on $\XX(\mu)$. Then, there exists a set $X \subseteq \XX(\mu)$ with the following properties:
\begin{itemize}
  \item[(a)]  $\mm(X)=1$.
  \item[(b)] For all $\omega \in X$ the autocorrelation $\gamma^{}_{\omega}$ exists with respect to $\cA$ and
  \[
  \gamma_{\omega}^{} \, = \, \gamma_{\mm}^{} \,.
  \]
  \item[(c)] For all $\omega \in X$ and all $\chi \in \widehat{G}$ the Fourier--Bohr coefficient $a_{\chi}^\cA(\omega)$ exists and satisfies the (CPP)
  \begin{equation}\label{eq1}
  \reallywidehat{\gamma_{\mm}^{}}(\{\chi\}) \, = \, \left| a_\chi^\cA(\omega) \right|^2 \,.
  \end{equation}
\end{itemize}
\end{corollary}
\begin{proof} Let $Y$ be the set of elements in $\XX(\mu)$ which are generic for $\mm$. Since $\XX(\mu)$ is compact and metrisable, by the ergodic theorem we have
\[
\mm(Y)=1 \,.
\]
By Theorem~\ref{thm:BL}, for all $\omega \in Y$, the autocorrelation $\gamma_{\omega}$ of $\omega$ exists along $\cA$ and
\[
\gamma_{\omega} \, = \, \gamma_{\mm} \,.
\]
Next, the set
\[
\BB \, :=\,  \{ \chi \in \widehat{G} \,:\,  \reallywidehat{\gamma_{\mm}}(\{ \chi \}) \neq 0\}
\]
is locally countable. Since $G$ is second countable, so is $\widehat{G}$ and hence $\BB$ is countable.

Now, by Proposition~\ref{PropBT}, for all $\chi \notin \BB$ and all $\omega \in Y$, the Fourier--Bohr coefficient $a_\chi^\cA(\omega)$ exists and
\[
\reallywidehat{\gamma_{\mm}^{}}(\{\chi\})\,=\, 0 \,=\,  \left| a_\chi^\cA(\omega) \right|^2 \,.
\]

Next, for each $\chi \in \BB$,  by Theorem~\ref{Thm-BTDS} there exists a set $X_\chi \subseteq \XX(\mu)$ such that for all  $\omega \in Y_\chi$ the Fourier--Bohr coefficient $a_\chi^\cA(\omega)$ exists and satisfies
\eqref{eq1}. Since $\BB$ is countable, the set
\[
X \, := \, Y \cap \left(\bigcap_{\chi \in \BB}^{} Y_ \chi \right)
\]
has full measure, and the claim follows.
\end{proof}

\begin{remark} If $\XX(\mu)$ is uniquely ergodic, then by Corollary~\ref{cor22} the autocorrelation $\gamma^{}_{\omega}$ exists with respect to $\cA$ and satisfies
  \[
  \gamma_{\omega}^{} \, = \, \gamma_{\mm}^{} \,,
  \]
  for all $\omega \in \XX(\mu)$. On another hand, even in this situation, the (CPP) is guaranteed only to hold almost surely. An example of an uniquely ergodic dynamical system where the (CPP) does not hold for all $\omega \in \XX(\mu)$ is given in \cite[Appendix~A]{LSS}. \exend
\end{remark}

Corollary~\ref{cor4} immediately gives the following result.

\begin{corollary} Let $\mu \in \cM^\infty(G)$, let $\cA$ be a van Hove sequence along which the Birkhoff ergodic theorem holds and let $\mm$ be an ergodic measure on $\XX(\mu)$.
Then,
\begin{itemize}
  \item[(a)] The following are equivalent:
  \begin{itemize}
    \item[(i)] The diffraction spectrum $\reallywidehat{\gamma^{}_{\mm}}$ of $(\XX(\mu), \mm)$ is continuous.
    \item[(ii)] For $\mm$-almost all $\omega \in \XX(\mu)$ the autocorrelation $\gamma^{}_{\omega}$ exists with respect to $\cA$ and  $\reallywidehat{\gamma^{}_{\omega}}$  is continuous.
    \item[(iii)] There exists a set $X \subseteq \XX(\mu)$ of full measures such that, for all $\omega \in X$ and all $\chi \in \widehat{G}$ the Fourier--Bohr coefficient $a_{\chi}^\cA(\omega)$ exists and is zero.
  \end{itemize}
  \item[(b)] The following are equivalent:
  \begin{itemize}
    \item[(i)] The Bartlett diffraction spectrum $\Gamma^{}_{\mm}$ of $(\XX(\mu), \mm)$ is continuous.
    \item[(ii)] For $\mm$-almost all $\omega \in \XX(\mu)$ the autocorrelation $\gamma^{}_{\omega}$ exists with respect to $\cA$ and the Bartlett diffraction spectrum $\Gamma^{}_{\omega}$  is continuous.
    \item[(iii)] There exists a set $X \subseteq \XX(\mu)$ of full measures such that, for all $\omega \in X$ and all $0 \neq \chi \in \widehat{G}$ the Fourier--Bohr coefficient $a_{\chi}^\cA(\omega)$ exists and is zero.. \qed
  \end{itemize}
\end{itemize}
\end{corollary}

\medskip

Let us complete the section by discussing the connection between the uniform existence of the Fourier--Bohr coefficients and continuity of the eigenfunctions.
The following result was basically proved in \cite[Theorem~6.7]{LSS}, but only partially stated. We include the short proof for completion.

\begin{proposition}\label{prop2} Let $\mu \in \cM^{\infty}_{}(G)$ and $\chi \in \widehat{G}$. Then, the Fourier--Bohr coefficient $a^{}_\chi(\mu)$ exists uniformly if and only if the Fourier--Bohr coefficient $a^{}_\chi(\omega)$ exists uniformly for all $\omega \in \XX(\mu)$.

Moreover, in this case the function
\[
\XX(\mu) \ni \omega \to a^{}_\chi(\omega) \in \CC
\]
is continuous and satisfies
\[
a^{}_\chi(\tau^{}_t\omega) \, = \, \overline{\chi(t)} \cdot a^{}_\chi(\omega)  \qquad \forall t \in G , \omega \in \XX(\mu) \,.
\]
\end{proposition}
\begin{remark} If $a^{}_\chi(\mu) \neq 0$ then
\[
f_\chi(\omega) \, := \, \overline{ a^{}_\chi(\omega) }
\]
is a continuous eigenfunction with eigenvalue $\chi$ \cite{LSS}.

On another hand, in this paper we are mainly interested in the case  $a^{}_\chi(\mu) = 0$. \exend
\end{remark}
\begin{proof}
$\bf \Longleftarrow$ is obvious.

\smallskip \noindent $\bf \Longrightarrow$ We follow closely the proof of \cite[Theorem~6.7]{LSS}. 

Fix one $\varphi \in \Cc(G)$ with $\widehat{\varphi}(\chi)=1$.

For each $n$, define $A^{}_n : \XX(\mu) \to \CC$ via
\[
A^{}_n(\omega) \, := \, \frac{1}{|A_n|} \int_{A_n}^{} \overline{\chi(t)} \, \omega*\varphi(t) \dd t \,.
\]
It is easy to see that $A^{}_n$ is continuous. Indeed, since the vague topology is metrisable on $\XX(\mu)$ we can work with sequences to show continuity. Let $\omega^{}_m, \omega \in \XX(\mu)$ be so that $\omega_{m}^{} \to \omega$ vaguely.
Then,
\[
\omega^{}_m*\varphi \, \stackrel{\mbox{pointwise}}{\relbar\joinrel\relbar\joinrel\relbar\joinrel\relbar\joinrel\relbar\joinrel\longrightarrow} \, \omega*\varphi \,.
\]
Moreover, by equi-translation boundedness, there exists a constant $C$ so that
\begin{align*}
  \| \omega^{}_m*\varphi \|_\infty \, &\leq \, C \qquad \forall n  \\
  \| \omega*\varphi \|_\infty \, &\leq \, C \qquad \forall n  \,.
\end{align*}
The Lebesgue dominated convergence theorem then gives
\[
\lim^{}_{m \to \infty} A^{}_n(\omega^{}_m) \,= \,  A^{}_n(\omega) \,.
\]
This shows that $A^{}_n$ is indeed continuous on $\XX(\mu)$.

Next, let us note that the uniform convergence of the Fourier--Bohr coefficients gives
\[
\lim_{n \to \infty}^{} \sup_{t \in G} \left|  A^{}_n(\tau_t \mu) - \overline{\chi(t)} a_\chi(\mu) \right| \, =\, 0 
\]
uniformly in $t$.

This implies that for each $\eps >0$ there exists some $N$ so that for all $m,n >N$ we have
\[
\left|A^{}_n(\tau_t \mu) - A^{}_m(\tau_t \mu) \right| \, \leq \,  \eps \qquad \forall t \in G \,.
\]
Since $A^{}_n, A^{}_m$ are continuous and 
\[
O^{}_\mu \, := \, \bigl\{ \tau^{}_t \mu \, : \, t \in G \bigr\} 
\] 
is dense in $\XX(\mu)$ we get
\[
\|A^{}_n - A^{}_m \|_\infty \, \leq \,  \eps \qquad \forall m,n >N \,.
\]
Therefore, $A^{}_n$ is Cauchy in $(\XX(\mu), \| \, \|_\infty)$ and hence convergent to a continuous function $F \in C(\XX(\mu))$.

This shows that for all $\omega \in \XX(\mu)$ we have
\[
F(\omega) \, = \, \lim_{n \to \infty} A^{}_n(\omega) \,=\,  a_\chi(\omega*\varphi)
\]
uniformly in $\omega$. In particular, this convergence is uniform on 
\[
O^{}_{\omega} \,:= \bigl\{ \tau_t^{} \omega  \, : \,  t \in G \bigr\} \, \subseteq \XX(\mu) \,.
\]

This implies that the Fourier--Bohr coefficient $a^{}_\chi(\omega*\varphi) $ exists uniformly, and hence, by Lemma~\ref{Lem-FB}, the Fourier--Bohr coefficient $a_\chi^{}(\omega)$ exists uniformly. Moreover,
\[
a_\chi^{}(\omega) \, = \, \widehat{\varphi(\chi)} a_\chi^{}(\omega) \, =\, a_\chi(\omega*\varphi) =F(\omega) \,.
\]
The claim and the statement after "moreover" follow.
\end{proof}

As an immediate consequence, we get:

\begin{corollary} Let $\mu \in \cM^{\infty}_{}(G)$ and $\chi \in \widehat{G}$. Then, the Fourier--Bohr coefficient $a^{}_\chi(\mu)$ is uniformly vanishing if and only if $a^{}_\chi(\omega)$ is uniformly vanishing for all $\omega \in \XX(\mu)$. \qed
\end{corollary}

\section{Uniform vanishing of the Fourier--Bohr coefficient}

In this section we give various characterisations of the uniform vanishing of the Fourier--Bohr coefficients.

\begin{theorem}\label{T1} Let $\mu \in \cM^\infty(G)$ and let $\cA$ be a van Hove sequence. Let $\chi \in \widehat{G}$. Then, the following are equivalent:
\begin{itemize}
\item[(i)] The Fourier--Bohr coefficient $a_\chi^{\cA}(\mu)$ are uniformly vanishing along $\cA$.
\item[(ii)] For all van Hove sequence $\cB$, Fourier--Bohr coefficient $a_\chi^{\cB}(\mu)$ is uniformly vanishing.
\item[(iii)]  The Fourier--Bohr coefficient $a_\chi^{\cA}(\omega)$ is uniformly vanishing for all $\omega \in \XX(\mu)$.
\item[(iv)] For each van Hove sequence $\cB$ the following Fourier--Bohr coefficient exists and
\[
a_\chi^\cB(\mu) \,=\, 0 \,.
\]
\item[(v)]If the Fourier--Bohr coefficient $a_\chi^\cB(\mu)$ exists along some van Hove sequence $\cB$ then
  \[
a_\chi^\cB(\mu) \,= \, 0 \,.
\]
\item[(vi)] For every autocorrelation $\gamma^{}_{\mu}$ of $\mu$ with respect to some van Hove sequence $\cB$ we have
  \[
  \reallywidehat{\gamma^{}_{\mu}}(\{ \chi \})\,=\,0 \,.
  \]
\item[(vii)] For all $G$-invariant probability measures $\mm$ on $\XX(\mu)$ we have
\[
\reallywidehat{\gamma_{\mm}}(\{ \chi \}) \,=\, 0 \,.
\]
\item[(viii)] For all ergodic measures $\mm$ on $\XX(\mu)$ we have
\[
\reallywidehat{\gamma_{\mm}}(\{ \chi \})\,=\, 0 \,.
\]
\item[(ix)] For every $\omega \in \XX(\mu)$ and every autocorrelation $\gamma_{\omega}^{}$ of $\omega$ along some van Hove sequence $\cB$ we have
  \[
  \reallywidehat{\gamma^{}_{\omega}}(\{ \chi \})\,=\,0 \,.
  \]
\end{itemize}
\end{theorem}
\begin{proof} (i) $\Longleftrightarrow$ (v)  $\Longleftrightarrow$ (vi) and (iii)  $\Longleftrightarrow$ (ix) follow from Proposition~\ref{Prop1}.

\smallskip \noindent (i) $\Longleftrightarrow$ (ii)  $\Longleftrightarrow$ (iv) follows from Corollary~\ref{cor2}.

\smallskip \noindent (i) $\Longleftrightarrow$ (iii) follows from Proposition~\ref{prop2}.

\smallskip \noindent(iii) $\Longrightarrow$ (viii)

Let $\mm$ be an ergodic measure on $\XX(\mu)$. By \cite[Theorem~5(b)]{DL} for almost all $\omega \in \XX(\mu)$ the Fourier--Bohr coefficient $a_\chi^\cA(\omega)$ exists and
\[
\reallywidehat{\gamma_\mm}(\{ \chi \}) \,=\, |a_\chi^\cA(\omega)|^2  \,.
\]
By (iii) we have $a_\chi^\cA(\omega)=0$ and (viii) follows.

\smallskip \noindent(viii) $\Longrightarrow$ (vii) follows from \cite[Corollary~7.16]{FRS}.

\smallskip \noindent(vii) $\Longrightarrow$ (vi) follows from Corollary~\ref{Cor1}.

\end{proof}

\medskip

Let us next state the following immediate consequences of Theorem~\ref{T1}.  The first result is a stronger version of \cite[Corollary~4.2]{NS}, as it removes the strong assumption on the uniform existence of Fourier--Bohr coefficients. It also slightly improves \cite[Corollary~7.8]{LS}.

\begin{theorem}\label{T11} Let $\mu \in \cM^\infty(G)$. Then, the following are equivalent:
\begin{itemize}
\item[(i)] For every autocorrelation $\gamma^{}_{\mu}$ of $\mu$ with respect to some van Hove sequence $\cB$, the diffraction spectrum $\reallywidehat{\gamma^{}_{\mu}}$ is continuous.
\item[(ii)] The Fourier--Bohr coefficient $a_\chi(\mu)$ is uniformly vanishing for all $\chi \in \widehat{G}$.
\item[(iii)]  The Fourier--Bohr coefficient $a_\chi^{\cA}(\omega)$ is uniformly vanishing for all $\omega \in \XX(\mu)$ and all $\chi \in \widehat{G}$.
\item[(iv)] For each van Hove sequence $\cB$ and all $\chi \in \widehat{G}$ the following Fourier--Bohr coefficient exists and
\[
a_\chi^\cB(\mu) \,=\, 0 \,.
\]
\item[(v)]If $a_\chi^\cB(\mu)$ exists for some van Hove sequence $\cB$ and some $\chi \in \widehat{G}$ then
  \[
a_\chi^\cB(\mu) \,=\, 0 \,.
\]
\item[(vi)] For all $G$-invariant probability measures $\mm$ on $\XX(\mu)$, the diffraction measure $\reallywidehat{\gamma_{\mm}}$ is continuous.
\item[(vii)] For all ergodic measures $\mm$ on $\XX(\mu)$, the diffraction measure $\reallywidehat{\gamma_{\mm}}$ is continuous.
\item[(viii)] For every $\omega \in \XX(\mu)$ and every autocorrelation $\gamma_{\omega}^{}$ of $\omega$ along some van Hove sequence $\cB$, the diffraction $\reallywidehat{\gamma^{}_{\omega}}$ is continuous. \qed
\end{itemize} 
\end{theorem}

\begin{theorem}\label{T111}  Let $\mu \in \cM^\infty(G)$. Then, the following are equivalent:
\begin{itemize}
\item[(i)] For every autocorrelation $\gamma_{\mu}$ of $\mu$ with respect to some van Hove sequence $\cB$, the Bartlett diffraction spectrum $\Gamma_\mu$ is continuous.
\item[(ii)] The Fourier--Bohr coefficient $a_\chi(\mu)$ is uniformly vanishing for all $0 \neq \chi \in \widehat{G}$.
\item[(iii)]  The Fourier--Bohr coefficient $a_\chi^{\cA}(\omega)$ is uniformly vanishing for all $\omega \in \XX(\mu)$ and all $0 \neq \chi \in \widehat{G}$.
\item[(iv)] For each van Hove sequence $\cB$ and all $0 \neq \chi \in \widehat{G}$ the following Fourier--Bohr coefficient exists and
\[
a_\chi^\cB(\mu) \,=\, 0 \,.
\]
\item[(v)]If $a_\chi^\cB(\mu)$ exists for some van Hove sequence $\cB$ and some $0 \neq \chi \in \widehat{G}$ then
  \[
a_\chi^\cB(\mu)\,=\, 0 \,.
\]
\item[(vi)] For each $G$-invariant probability measures $\mm$ on $\XX(\mu)$ the Bartlett diffraction $\Gamma^{}_\mm$ is continuous. 
\item[(vii)] For each ergodic measures $\mm$ on $\XX(\mu)$ the Bartlett diffraction $\Gamma^{}_\mm$ is continuous. 
\item[(viii)] For every $\omega \in \XX(\mu)$ and every autocorrelation $\gamma_{\omega}^{}$ of $\omega$ along some van Hove sequence $\cB$, the Bartlett diffraction $\Gamma^{}_{\omega}$ is continuous. \qed
\end{itemize} 
\end{theorem}

\medskip 

In the case when $\mu$ has an unique autocorrelation, and in particular when $\XX(\mu)$ is uniquely ergodic, we get the following immediate consequences:

\begin{corollary} Let $\mu \in \cM^\infty(G)$ be a measure which has an unique autocorrelation $\gamma$. Then, the following are equivalent:
\begin{itemize}
\item[(i)] The diffraction spectrum $\reallywidehat{\gamma}$ is continuous.
\item[(ii)] The Fourier--Bohr coefficient $a_\chi(\mu)$ are uniformly vanishing for all $\chi \in \widehat{G}$.
\item[(iii)]  The Fourier--Bohr coefficient $a_\chi^{\cA}(\omega)$ are uniformly vanishing for all $\omega \in \XX(\mu)$ and all $\chi \in \widehat{G}$.
\item[(iv)] For each van Hove sequence $\cB$ and all $\chi \in \widehat{G}$ the following Fourier--Bohr coefficient exists and
\[
a_\chi^\cB(\mu) \,=\, 0 \,.
\]
\item[(v)]If $a_\chi^\cB(\mu)$ exists along some van Hove sequence $\cB$ for some $\chi \in \widehat{G}$ then
  \[
a_\chi^\cB(\mu) \,=\, 0 \,.\tag*{$ \qed $}
\]
\end{itemize} 
\end{corollary}

\begin{corollary} Let $\mu \in \cM^\infty(G)$ be a measure which has an unique autocorrelation $\gamma$. Then, the following are equivalent:
\begin{itemize}
\item[(i)] The Bartlett diffraction spectrum $\Gamma$ of $\mu$  is continuous.
\item[(ii)] The Fourier--Bohr coefficient $a_\chi(\mu)$ are uniformly vanishing for all $0 \neq \chi \in \widehat{G}$.
\item[(iii)]  The Fourier--Bohr coefficient $a_\chi^{\cA}(\omega)$ are uniformly vanishing for all $\omega \in \XX(\mu)$ and all $0 \neq \chi \in \widehat{G}$.
\item[(iv)] For each van Hove sequence $\cB$ and all $0 \neq \chi \in \widehat{G}$ the following Fourier--Bohr coefficient exists and
\[
a_\chi^\cB(\mu) \,=\, 0 \,.
\]
\item[(v)]If $a_\chi^\cB(\mu)$ exists along some van Hove sequence $\cB$ for some $0 \neq \chi \in \widehat{G}$ then
  \[
a_\chi^\cB(\mu) \,=\, 0 \,.\tag*{$ \qed $}
\]
\end{itemize} 
\end{corollary}

\section{Continuous dynamical spectrum}

In this section we discuss the continuity of the dynamical spectrum. Let us start with the following result, which generalizes \cite[Lemma~4.6]{LS3} and characterizes the eigenvalues of the Dynamical system. The result can be seen as a dynamical spectrum version of Theorem~\ref{T1}.

\begin{theorem} Let $(X,G)$ be a metric dynamical system and let $\cA$ be a F\o lner sequence. Let $\chi \in \widehat{G}$. Then, the following are equivalent:
\begin{itemize}
  \item[(i)] For each $G$-invariant measure $\mm$ on $X$, $\chi$ is not an eigenvalue for $\mm$.
  \item[(ii)] For each ergodic measure $\mm$ on $X$, $\chi$ is not an eigenvalue for $\mm$.
  \item[(iii)]  For each $f\in C(X)$ and any sequence $\chi_n \in \widehat{G}$ converging pointwise to $\chi$  we have
\[
\lim_n \frac{1}{|A_n|} \int^{}_{A_n}  f(t.x) \cdot \overline{\chi_n(t)} \, \dd t \,=\, 0
\]
uniformly in $x\in X$.
  \item[(iv)] For each $f\in C(X)$ we have
\[
\lim_n \frac{1}{|A_n|} \int^{}_{s+A_n}  f(t.x) \cdot \overline{\chi(t)} \, \dd t \,= \, 0
\]
uniformly in $x\in X, s \in G$.
  \item[(v)] The set of functions $f \in C(X)$ with the property that
\[
\lim_n \frac{1}{|A_n|} \int^{}_{s+A_n}  f(t.x) \cdot \overline{\chi(t)} \, \dd t \,=\, 0
\]
uniformly in $x\in X, s \in G$ is dense in $C(X)$.
\end{itemize}
\end{theorem}
\begin{proof}
The equivalence (i)$\Longleftrightarrow$ (ii)$\Longleftrightarrow$ (iii) is  \cite[Lemma~4.6]{LS3}.

\smallskip \noindent(iii) $\Longrightarrow$ (iv) is obvious. Indeed, we can get (iv) from (iii) by setting $\chi_n:= \chi$ and replacing $x$ by $s.x$.

\smallskip \noindent(iv) $\Longrightarrow$ (ii) is mot-a-mot the proof of the implication (iii) $\Longrightarrow$ (ii) in  \cite[Lemma~4.6]{LS3}.

\smallskip \noindent(iv) $\Longrightarrow$ (v) is obvious.

\smallskip \noindent(v) $\Longrightarrow$ (iv) follows immediately from the observation that for all $f,g \in C(X)$, all $x \in X, s \in G$ and all $n$ we have
\[
\left| \frac{1}{|A_n|} \int_{s+A_n}  f(t.x) \cdot \overline{\chi(t)} \, \dd t -\frac{1}{|A_n|} \int_{s+A_n}  g(t.x) \cdot \overline{\chi(t)} \, \dd t  \right| \, \leq\, \|f-g \|_\infty \,. \qedhere
\]
\end{proof}

\medskip 

Now, fix some $\mu \in \cM^\infty(G)$. It is immediate that the algebra $\AAA$ spanned by
\[
\{ N_\varphi, \overline{N_\varphi} \,:\,  \varphi \in \Cc(G) \} \cup \{ 1 \}
\]
is separating the points of $C(\XX(\mu))$ and hence is dense in $\Cc(\XX(\mu))$. This immediately implies:

\begin{corollary}\label{C1} Let $\mu \in \cM^\infty(G)$, let $\chi \in \widehat{G}$ and let $\cA$ be a van Hove sequence. Then, the following are equivalent:
\begin{itemize}
  \item[(i)] For each $G$-invariant measure $\mm$ on $\XX(\mu)$, $\chi$ is not an eigenvalue for $\mm$.
  \item[(ii)] For each ergodic measure $\mm$ on $\XX(\mu)$, $\chi$ is not an eigenvalue for $\mm$.
  \item[(iii)] For each $\varphi_1, \ldots, \varphi_m, \psi_1, \ldots, \psi_l \in \Cc(G)$ the Fourier--Bohr coefficient $a_\chi(f)$ of
  \[
  f(t)= \left( \prod_{j=1}^m \mu*\varphi_j(t) \right) \overline{ \left( \prod_{k=1}^l \mu*\psi_j(t) \right)} 
  \]
  are uniformly vanishing. \qed
\end{itemize}
\end{corollary}

Also, as an immediate consequence of Corollary~\ref{C1} we get:

\begin{theorem}Let $\mu \in \cM^\infty(G)$. Then, the following, are equivalent:
\begin{itemize}
\item[(i)] Each $G$-invariant measure $\mm$ on $\XX(\mu)$ has continuous dynamical spectrum.
  \item[(ii)] Each ergodic measure $\mm$ on $\XX(\mu)$ has continuous dynamical spectrum.
  \item[(iii)] For all $\varphi_1, \ldots, \varphi_m, \psi_1, \ldots, \psi_l \in \Cc(G)$ and all $0 \neq \chi \in \widehat{G}$, the Fourier--Bohr coefficients $a_\chi(f)$ of
  \[
  f(t)= \left( \prod_{j=1}^m \mu*\varphi_j(t) \right) \overline{ \left( \prod_{k=1}^l \mu*\psi_j(t) \right)}
  \]
are uniformly vanishing. \qed
\end{itemize}
\end{theorem}

Let us complete the paper by briefly reviewing one classical example of a measure with continuous diffraction spectra but mixed dynamical spectrum. In particular, in this case, for each $\varphi \in \Cc(G)$ the function $\mu*\varphi$ must have uniformly null Fourier--Bohr coefficients, but some products of such functions have non-zero Fourier--Bohr coefficients.

\begin{example}[Thue--Morse] Consider the $\{ \pm 1 \}$ weighted Dirac comb $\omega$ of the Thue--Morse sequence, see \cite[Section~4.6]{TAO} for details.

Then $\XX(\omega)$ is uniquely ergodic (see for example \cite{Kak,BC}) and its unique autocorrelation is purely singular continuous \cite[Theorem 10.1]{TAO}.

On another hand, the set of eigenvalues of $\XX(\omega)$ is $\ZZ[\frac{1}{2}]$ and each eigenvalue is topological, meaning it has a continuous eigenfunction \cite{Que}. \exend
\end{example}

\subsection*{Acknowledgments}
We would like to thank Michael Baake, for some discussions and suggestions which improved this manuscript. The author was supported by the Natural Sciences and Engineering Council of Canada via grant 2024-04853, and he is grateful for the support.


\begin{thebibliography}{99}

\bibitem{ARMA1}
L.~Argabright and J.~Gil de Lamadrid, \textit{Fourier Analysis of
Unbounded Measures on Locally Compact Abelian Groups}, Memoirs
Amer.\ Math.\ Soc. \textbf{145}, (AMS, Providence,
RI) (1974) .

\bibitem{TAO}
M. ~Baake and U.~ Grimm, \textit{Aperiodic Order. Vol.~1: A Mathematical Invitation},
(Cambridge University Press, Cambridge) (2013).

\bibitem{TAO2}
M. Baake and U. Grimm (eds.), \textit{Aperiodic Order. Vol.~2: Crystallography and Almost Periodicity}
(Cambridge University Press, Cambridge) (2017).

\bibitem{BBG}
M. Baake, M.~Birkner and U.~Grimm, \textit{Non-periodic systems with continuous diffraction measures}. In: \cite{KLS} (2015), pp. 1--32; \texttt{arXiv:1502.05122}.


\bibitem{BBM}
M. Baake, M.~Birkner and R.V.Moody, \textit{Diffraction of stochastic point sets: Explicitly computable examples},  Commun. Math. Phys. \textbf{293} (2010), 611--660; \texttt{arXiv:0803.1266}.


\bibitem{BC}
M. Baake and M. Coons, \textit{Correlations of the Thue--Morse sequence}, Indag. Math. \textbf{35(5)} (2024), 914--930; \texttt{arXiv:2209.07102}.



\bibitem{BL}
M. Baake and D. Lenz, \textit{Dynamical systems on translation bounded measures: pure point dynamical and diffraction spectra},
Erg. Th. \& Dynam. Syst. \textbf{24}(2004), 1867--1893; \texttt{arXiv:math/0302061}.

\bibitem{BM}
M.~Baake and R.V.~Moody (eds.), \textit{Directions in Mathematical
Quasicrystals}, CRM Monograph Series \textbf{13}, AMS (Providence, RI)
(2000).


\bibitem{BF}
C. ~Berg and G. ~Forst, \emph{Potential Theory on Locally Compact
Abelian Groups}, (Springer, Berlin) (1975).

\bibitem{FRS}
M. Francis, C. Ramsey and N. Strungaru,
\textit{On the spectrum of non-ergodic measures},  preprint (2026); \texttt{arXiv:2601.05327}.

\bibitem{ARMA}
J.~Gil.~de~Lamadrid and L.~N.~Argabright, \textit{Almost Periodic Measures}, Memoirs Amer. Math. Soc. \textbf{85(428)}, (AMS, Providence,
RI)(1990).

\bibitem{Gou}
J.-B.~Gou\'{e}r\'{e}, \textit{Diffraction and Palm measure of point processes}, C. R. Acad. Sci. Paris \textbf{336} (2003), 57--62; \texttt{arXiv:math/0208064}.


\bibitem{Gou2}
J.-B. Gou\'{e}r\'{e}, \textit{Quasicrystals and almost
periodicity}, Commun. Math. Phys. \textbf{255} (2005), 655--681; \texttt{arXiv:math-ph/0212012}.

\bibitem{HOF}
A. Hof, \textit{On diffraction by aperiodic structures}, Commun. Math. Phys. \textbf{169}(1995), 25--43.

\bibitem{HOF1}
A. Hof, \textit{Diffraction by aperiodic structures}. In: \textit{The Mathematics of Long-Range Aperiodic Order}(R.V. Moody ed.), NATO-ASI Series C \textbf{489} (Kluwer, Dordrecht) (1997), pp. 239--268.



\bibitem{Kak}
S. Kakutani, \textit{Strictly ergodic symbolic dynamical systems}. In: \textit{Proceedings of the Sixth Berkeley Symposium
on Mathematical Statistics and Probability}(L.M. Le Cam, J. Neyman and E.L. Scott eds.), University of
California Press, (Berkeley) (1972), pp. 319--326.

\bibitem{KLS}
J.~Kellendonk, D.~Lenz and J.~Savinien (eds.), \textit{Mathematics of Aperiodic Order},  Progr. Math. \textbf{309} (Birkh\"auser/Springer, Basel) (2015).


\bibitem{LMS}
J.-Y.\ Lee, R.~V.~Moody and  B.~Solomyak, \textit{ Pure point dynamical and diffraction spectra}, Ann. H.\ Poincar\'{e} \textbf{3} (2002), 1003--1018; \texttt{arxiv:0910.4809}.

\bibitem{DV}
D.D.~Daley, D.~Vere-Jones, \textit{ An Introduction to the Theory of Point Processes I: Elementary Theory and
Methods.} 2nd ed., 2nd corr. printing, (New York: Springer)(2005).


\bibitem{DL}
D. Lenz, \textit{Continuity of eigenfunctions of uniquely ergodic dynamical systems and intensity of Bragg peaks},
Commun. Math. Phys. \textbf{287(1)} (2009), 225--258;
\texttt{arXiv:math-ph/0608026}.

\bibitem{LSS}
D.~Lenz, T.~Spindeler and N.~Strungaru, \textit{Pure point diffraction and mean, Besicovitch and Weyl almost periodicity},
\textit{preprint} (2020); \texttt{arXiv:2006.10821}.


\bibitem{LSS3}
D. Lenz, T. Spindeler and N. Strungaru, \textit{The (reflected) Eberlein convolution of measures}, Indag. Math. \textbf{35(5)} (2023), 959--988; \texttt{arXiv:2211.06969}.


\bibitem{LS3}
D. Lenz and N. Strungaru, \textit{Wiener--Wintner points for topological dynamical systems},  preprint (2025); \texttt{arXiv:2510.18056}.


\bibitem{LS}
D. Lenz and N. Strungaru, \textit{ Diffraction as a unitary representation and the orthogonality of measures with respect to the reflected Eberlein convolution}, to appear in J.  Four. Anal. and App. (2026); \texttt{arXiv:2402.01044}.

\bibitem{MPS}
R.V.~Moody, D.~Postnikoff and N. Strungaru, \textit{Circular symmetry of pinwheel diffraction}, Ann. H.\ Poincar\'{e} \textbf{7(4)} (2006), 711--730.

\bibitem{MS}
R.V.~Moody and N.~Strungaru, \textit{ Point sets and dynamical systems in the autocorrelation topology}, Canad. Math. Bull. \textbf{47(1)}(2004), 82--99.  	


\bibitem{MoSt}
R.V.~Moody and N.~Strungaru, \textit{ Almost periodic measures and their Fourier
transforms}. In: \cite{TAO2}(2017), pp. 173--270.


\bibitem{Que}
M. Queff\'elec, \textit{Substitution Dynamical Systems—Spectral Analysis}, LNM \textbf{1294} 2nd ed., (Springer, Berlin) (2010).


\bibitem{She}
D.~Shechtman, I.~Blech, D.~Gratias and J.W.~Cahn,  \textit{Metallic phase with
long-range orientational order and no translational symmetry}, Phys. Rev. Lett.
\textbf{53} (1984), 1951--1953.


\bibitem{Sch}
M.~Schlottmann, \textit{Generalized model sets and dynamical
systems}. In: \cite{BM}(2002), pp 143--159.



\bibitem{NS}
N. Strungaru, \textit{On the orthogonality of measures of different spectral type with respect to twisted Eberlein convolution}, preprint (2022); \texttt{arXiv:2402.01044}.


\end{thebibliography}
\end{document}